\documentclass[a4paper,11pt]{amsart}
\usepackage[british]{babel}
\usepackage[colorlinks=true,citecolor=black,linkcolor=black]{hyperref}
\usepackage{cite,tikz}

\newtheorem{bigthm}{Theorem}

\newtheorem{lem}{Lemma}[section]
\newtheorem*{QuasiComp}{Theorem\,[\citen{FT15},\,1.3.1]}
\newtheorem*{Bands}{Theorem\,[\citen{FT15},\,1.3.4]}
\newtheorem*{NumRes}{Theorem\,[\citen{FT15},\,1.3.8]}
\theoremstyle{remark}
\newtheorem{rem}{Remark}

\newcommand{\abs}[1]{\left|#1\right|} 
\newcommand{\id}{\operatorname{id}} 
\newcommand{\cB}{\mathcal{B}} 
\newcommand{\cC}{\mathcal{C}} 
\newcommand{\cD}{\mathcal{D}} 
\newcommand{\bC}{\mathbb{C}} 
\newcommand{\cH}{\mathcal{H}} %
\newcommand{\bT}{\mathbb{T}} 
\newcommand{\cJ}{\mathcal{J}} %
\newcommand{\bN}{\mathbb{N}} 
\newcommand{\norm}[1]{\left\|#1\right\|} 
\newcommand{\bR}{\mathbb{R}} 
\newcommand{\bZ}{\mathbb{Z}} 
\newcommand{\cS}{\mathcal{S}} %
\newcommand{\cO}{\operatorname{O}} 
\newcommand{\cW}{\mathcal{W}} 
\newcommand{\supp}{\operatorname{Supp}} 
\newcommand{\htop}{{{h}_{\text{top}}}} 
\newcommand{\Vol}{\operatorname{Vol}} 

\begin{document}
\title[Renormalizable Parabolic Flows]{Parabolic Flows Renormalized by Partially Hyperbolic Maps}
\date{}

\author{Oliver Butterley}
\address{(Oliver Butterley) University of Nova Gorica -- Vipavska 13 --  Nova Gorica -- 5000 -- Slovenia}
\email{oliver.butterley@ung.si}

\author{Lucia D. Simonelli}
\address{(Lucia  D. Simonelli) Abdus Salam International Centre for Theoretical Physics -- Strada Costiera 11 -- Trieste -- 34151 -- Italy}
\email{lucia.simonelli@ictp.it}

\begin{abstract}
We consider parabolic flows on 3-dimensional manifolds which are renormalized by circle extensions of Anosov diffeormorphisms. This class of flows includes nilflows on the Heisenberg nilmanifold which are renormalized by partially hyperbolic automorphisms. 
The transfer operators associated to the renormalization maps,   
acting on anisotropic Sobolev spaces, are known to have good spectral properties (this relies on ideas which have some resemblance to representation theory but also apply to non-algebraic systems). 
The spectral information is used to describe the deviation of ergodic averages and solutions of the cohomological equation for the parabolic flow. 
\end{abstract}

\maketitle
\thispagestyle{empty}

\section{Introduction}

Parabolic dynamical systems are roughly classified as a type of intermediate system: in hyperbolic systems orbits diverge exponentially, in elliptic systems there is no, or little to no, divergence and in parabolic systems the divergence of orbits of nearby points is polynomial in time. Examples of parabolic systems include horocycle flows on surfaces of negative curvature, flat flows on surfaces of higher genus and Heisenberg nilflows (see e.g., \cite{FF03,Forni02b,FF06} respectively). One of the most interesting and fruitful features of these systems is the intimate connection they enjoy with certain hyperbolic systems; this connection is referred to as {\em renormalization}. Let \(\Psi_t : M \to M\) denote the parabolic flow defined on a manifold \(M\). The most concrete instance of renormalization,  the case considered in this paper, is when there exists a map \(F: M \to M\), with some degree of hyperbolicity, which conjugates the parabolic system in the following way,
\begin{equation}
\label{eq:RenormFlow}
F \circ \Psi_{\lambda t} = \Psi_{t} \circ F,
\end{equation}
for all \(t\in \bR\) and \(\lambda>1\) fixed. (For a description of a more general scheme for renormalization, see \cite{Forni02a}.)

Renormalization is very useful, perhaps even vital, to proving important properties of these parabolic systems. Two such properties are: (i) deviation of ergodic averages; (ii) regularity of solutions to the cohomological equation. 
Given $x \in M$, \(t>0\), 
 the ergodic integral is defined, for \(h:M\to \bR\), as
\begin{equation} \label{eq:ErgodInt}
H_{x,t}(h) = \int_{0}^{t} h \circ \Psi_{r}(x) \> dr.
\end{equation}
Understanding the deviation of ergodic averages requires studying the behaviour, for \(h: M \to \bR\), of \(H_{x,t}(h)\) as \(t\to \infty\).
Solutions to the cohomological equation are functions \(g: M\to \bR\) such that  for a given \(h: M \to \bR\), 
\begin{equation}
\label{eq:Cohom}
g\circ \Psi_t(x) - g(x) = \int_{0}^{t} h \circ \Psi_r(x) \ dr
\end{equation}
for all \(x \in M\), \(t\geq 0\). It is then of interest to determine the regularity of the transfer function, \(g\).  
(The infinitesimal version of the cohomological equation is given by \(Wg = h\), where \(W\) denotes the vector field generating the flow \(\Psi_t\).)

Notice that both the above quantities, \eqref{eq:ErgodInt} \& \eqref{eq:Cohom}, involve integrals over increasingly longer orbits. By exploiting the renormalization of the parabolic flow by a hyperbolic system, these integrals can be transformed into equivalent expressions involving integration over orbits of a fixed length. This is the elementary, yet crucial, observation that, for any \(x \in M\), \(t>0\), \(h\in \cC^0(M)\), \(k\in \bN\),
\begin{equation}
\label{eq:IntRenorm}
\begin{split}
H_{x,t}(h) 
&=  \int_0^t  h \circ F^{-k} \circ F^{k} \circ \Psi_{r}(x) \> dr\\
&=   \int_0^t h \circ F^{-k}\circ \Psi_{\lambda^{-k}r}\circ F^{k}(x)  \> dr\\
&=   \int_0^{\lambda^{-k}t} \lambda^{k}   h \circ F^{-k} \circ \Psi_{u}(x_k)   \> du 
=  H_{x_k,\lambda^{-k}t} ( \widehat{F}^kh)
\end{split}
\end{equation}
where \(x_k:= F^{k}(x)\), 
\begin{equation}
\label{eq:TransOpGen}
\hat{F} : h \mapsto \lambda h \circ F^{-1}
\end{equation}
is the transfer operator, and \(\lambda^{-1}\) denotes the contraction of \(F\) in the stable direction and is consequently the same factor which appears in the renormalization~\eqref{eq:RenormFlow} relation (in general the leading term  \(\lambda \) is replaced by the inverse of the stable Jacobian of \(F\)).

By considering renormalization on the bundle of invariant distributions over the moduli space and employing tools from representation theory and harmonic analysis, information on deviation of ergodic averages, and in some cases, results on regularity of solutions to the cohomological equation, were obtained for various different systems. This was done by Forni~\cite{Forni97,Forni02b} in the case of area-preserving flows on surfaces of higher genus, Flaminio \& Forni~\cite{FF03} in the case of the horocycle flow on hyperbolic surfaces, and again by Flaminio \& Forni \cite{FF06} in the case of Heisenberg nilflows. In the case of interval exchange transformations (closely related to area-preserving flows), Mousa, Marmi \& Yoccoz~\cite{MMY05} were able to extend results from Forni~\cite{Forni97} to a larger class of systems by considering renormalization based on Rauzy-Veech induction. 

One of the most rewarding outcomes of the study of parabolic flows via renormalization was that it allowed for an explicit description of the obstructions to solving the cohomological equation (and obstructions to faster convergence of the ergodic integral) by extending beyond the space of functions and working instead in the space of distributions \cite{Forni02a, Forni02b, FF03}. This strongly suggested that there exists a link between the distributional obstructions and the spectra of transfer operators of hyperbolic systems acting on carefully chosen anisotropic Banach spaces (for an overview of anisotropic spaces see \cite{Baladi17} and references within).

The speculated connection between the distributions which arise studying parabolic flows and the distributions which arise studying hyperbolic systems was made solid in the work of Giulietti \& Liverani \cite{GL16}. They studied flows on the torus which are renormalized (in the sense of \eqref{eq:RenormFlow} but where \(\lambda t\) is instead  a more general function of \(x\) and \(t\)) by an Anosov map.
The map in their setting need not be linear, it is merely required that the foliations are \(\cC^{1+\alpha}\).
The key idea is to study the spectrum of the transfer operator associated to the hyperbolic map on an appropriate choice of anisotropic Banach space and then use this spectral information to answer the pertinent questions for the parabolic flow.
They show that indeed the eigenvalues of this operator correspond to the deviation spectrum of the ergodic averages as well as to the obstructions to solving the cohomological equation. 
The authors propose this method as a way of extending known results to more general settings, for example, beyond the smooth class of systems that have been studied using representation theory. 

Formalizing this connection has inspired a new wave of work on renormalization. This modified renormalization technique was applied by Faure, Gou\"ezel \& Lanneau \cite{FGL18} to the case of area-preserving flows on surfaces of higher genus renormalized by pseudo Anosov maps. The authors are able to link the obstructions to topological properties of the surface.
 Adam \cite{Adam18} applied this technique to the case of the horocycle flow on a surface of variable negative curvature renormalized by a geodesic flow, thus obtaining results which were previously only obtainable for constant curvature.
 
In the present work we study parabolic flows on \(3\)-dimensional manifolds which are renormalized by partially hyperbolic maps. 
The primary examples are Heisenberg nilflows renormalized by partially hyperbolic automorphisms of the nilmanifold which we describe in Section~\ref{sec:Heisenberg}. 
The main results of this work are contained in three theorems which we will now summarise. The precise statements appear in Sections~\ref{sec:ExtraFunc}, \ref{sec:ErgAver}, \ref{sec:Cohom} respectively once the required notation has been introduced.
Section~\ref{sec:Spectrum} is devoted to the description of partially hyperbolic maps and the spectral results concerning the associated transfer operators, derived by Faure \& Tsujii~\cite{FT15}. The anisotropic Sobolev spaces used are chosen carefully so that they are both large enough to provide useful information as well as precise enough to extract explicit spectral data. As a first step we prove that
\begin{quotation}
\itshape  \vspace{.2em}
The transfer operator \(\widehat F\), considered as an operator on appropriate anisotropic spaces, admits a countable spectral decomposition (Theorem~\ref{thm:SpecInfo}, Section~\ref{sec:ExtraFunc}). 
 \vspace{.2em}
\end{quotation}
This spectral decomposition is slightly unconventional in that it involves a countable sum and also that the sum does converge but not neccesarily in the obvious operator norm. 
Section~\ref{sec:ErgAver} is devoted to the study of parabolic flows under the assumption that they are renormalized, in the sense of~\eqref{eq:RenormFlow}, by a map which satisfies the spectral properties given by the theorems proved in Sections~\ref{sec:Spectrum} \& \ref{sec:ExtraFunc}. The method of studying the behavior of ergodic averages, $H_{x,t}(h)$ relies on two key parts. First we show that the ergodic integrals are well approximated by elements of the dual of our anisotropic Banach space. Then we compose the observable over long orbits of the flow, i.e., large values of $t$. The renormalization gives us a method by which we can study modified ergodic integrals that contain the same information but are taken over shorter orbits or, equivalently, over shorter intervals of time.
\begin{quotation}
\itshape \vspace{.2em}
If a parabolic flow is renormalized by a partially hyperbolic map which satisfies the spectral conclusions of Theorem~\ref{thm:SpecInfo} then we obtain a precise description of the deviation of ergodic averages (Theorem~\ref{thm:Deviation}, Section~\ref{sec:ErgAver}). 
 \vspace{.2em}
\end{quotation}
The argument used to estimate the deviation of ergodic averages is not dependent on the exact space of distributions used in the application of this paper and is consequently applicable to any other setting where similar spectral results are available for another Banach (or indeed Hilbert) space which satisfies a modest compatibility requirement (Lemma~\ref{lem:InDual}).
\begin{quotation}
\itshape  \vspace{.2em}
If a parabolic flow is renormalized by a partially hyperbolic map which satisfies the spectral conclusions of Theorem~\ref{thm:SpecInfo}, then the solutions of the cohomological equation exist and are equal to the uniform limit of an explicit quantity (Theorem~\ref{thm:Cobound}, Section~\ref{sec:Cohom}).
 \vspace{.2em}
\end{quotation}
This precise formula for the solutions to the cohomological equation allows for the possibility to further study the regularity of the solutions. A weight of evidence suggests that the natural extension of these ideas, following the notion of Giulietti \& Liverani~\cite{GL16} to study the action of the dynamics on one forms, would consequently provide information regarding the regularity question. In a different direction, Faure \& Tsujii~\cite[\S1.4]{FT15}  proved similar spectral results in the extension to the Grassmanian which should, building o the framework established in the present work, allow for all the analogous argument to be carried out in the case of flows along stable manifolds of partially hyperbolic maps like the ones studied in the present work but without a regularity condition on the invariant foliation. 

Theorem~\ref{thm:Deviation} and Theorem~\ref{thm:Cobound}, in terms of results, are not significant improvements over what was previously shown by Flaminio \& Forni~\cite{FF06}. The contribution of this work is to rigorously describe the connection between the asymptotic properties of nilflows and the spectral properties of partially hyperbolic dynamical systems, to refine the ideas for connecting parabolic dynamics and anisotropic spaces which was introduced by Giulietti \& Liverani~\cite{GL16} and most particularly, to explore generalized methods based on representation theory that allow for the possibility to progress in various different directions. 

\begin{rem}
The argument of Faure \& Tsujii \cite{FT15} could be described as an extension of the representation theoretic argument to the non-linear (non-algebraic) setting. In particular, the Bargmann Transform plays a central role in the definition of the norm. 
\end{rem}

\begin{rem}
As far as the authors know, the method of Faure \& Tsujii is the only one currently available that is able to give information about the inner bands of the spectrum. Seeing beyond the countable number of eigenvalues in the outer band is essential, particularly when studying the cohomological equation. Consequently, we expect that the spectral results and methods in Faure \& Tsujii \cite{FT13, FT17} for contact Anosov flows could prove useful for the study of properties of the horocycle flow.
\end{rem}

\subsection*{Acknowledgements}
It is a pleasure to thank Giovanni Forni and Carlangelo Liverani for suggesting the study of this subject and for numerous useful discussions. O.B. thanks Centro di Ricerca Matematica Ennio De Giorgi for hospitality during the event ``renormalization in Dynamics'' where the project was initiated. We are grateful to Alexander Adam, Mat\'ias Delgadino, Sebastien Gou\"ezel and Davide Ravotti for several helpful discussions and comments. We thank Centre International de Rencontres Math\'ematiques for hospitality during the event ``Probabilistic Limit Theorems for Dynamical Systems''.


\section{Circle Extensions of Symplectic Anosov Maps}
\label{sec:Spectrum}
In this section we describe the setting and the relevant spectral results obtained in the work of Faure \& Tsujii~\cite{FT15}.
Let \(S\) be a \(\cC^\infty\) closed connected  \(2\)-dimensional manifold with symplectic two form \(\omega\). Let \(F:S\to S\) be a linear Anosov map that preserves \(\omega\) (i.e., \(f^*\omega = \omega\)) and suppose that there exists an \(F\)-invariant smooth decomposition of \(TS\), \(T_xS = E_u(x) \oplus E_s(x)\) and a constant \(\lambda >1\) such that
\( \left.\det DF \right|_{E_u} = \lambda\) and \( \left.\det DF \right|_{E_s} = \lambda^{-1}\) (this notation denotes the Jacobian restricted to the respective subspace). 
The one-dimensional unitary group \(U(1)\) is the multiplicative group of complex numbers of the form \(e^{i\theta}\), \(\theta \in \bR\). Let \(P\) denote the \(U(1)\)-principle bundle over \(S\) and denote by \(\pi:P \to S\) the corresponding projection map. Furthermore we assume that there exists\footnote{In this setting, under mild assumptions \cite[Asm.\,1,2]{FT15} it was shown that such a connection and lift map always exists.} a connection \(A\) (which corresponds to the symplectic form) and a map (called the lift),
\[
 \widetilde{f}: P \to P
\]
such that 
\( (\pi \circ   \widetilde{f})(p) = (F \circ \pi)(p)\), 
\(F (e^{i\theta}) = e^{i\theta} F(p)\) for all \(p\in P\), \(\theta \in \bR\)
and \(F^* A = A\) where \(A\) is a connection whose curvature is \(-2\pi i \pi^* \omega\).
The map \(  \widetilde{f} : P \to P\) is partially hyperbolic in the sense that tangent bundle splits into three subbundles, one exponentially contracting, one exponentially expanding and a third, corresponding to the fibres of the bundle \(P\), in which the behaviour is neutral. The stable and unstable bundles are not jointly integrable in a precise way since \(F\) preserves a contact form~\cite[Remark 1.2.7]{FT15}.
For example, if \(S\) is taken to be the torus with the canonical symplectic form, then \(P\) is the Heisenberg nilmanifold. Taking \(F\) to be a toral automorphism means that the lift is an automorphism of the nilmanifold. Details of this particular example are given in Section~\ref{sec:Heisenberg}.

Corresponding to the quantity we require for parabolic flows~\eqref{eq:TransOpGen}, we consider the transfer operator \(\widehat{F}: \cC^{\infty}(P) \to \cC^{\infty}(P)\) defined as\footnote{This equates to the choice of potential \(V= \ln \lambda\) in the reference \cite{FT15}.}
\begin{equation}
\label{eq:TransOp}
\widehat{F}
: h \to
\lambda \, h \circ \tilde f^{-1}.
\end{equation}
The equivariance property of \(\tilde f\) in the fibres means that there is the following natural and useful decomposition of \(\cC^{\infty}(P)\) (the \(N\)\textsuperscript{th} element of the decomposition corresponds to the \(N\)\textsuperscript{th} Fourier mode in the fibre).
For \(N \in \bZ\) let
\[
\cC_N^{\infty}(P) : = \left\{h \in \cC^{\infty}(P) \ | \ h(e^{i \theta}p)= e^{iN \theta}h(p)
, \text{for all \(p\in P\), \(\theta \in \bR\)}\right\}
\]
and let \(\widehat{F}_N : \cC_N^{\infty}(P) \to \cC_N^{\infty}(P) \) be defined as
\begin{equation}
\label{eq:NTransOp}
\widehat{F}_N  := \left. \smash{\widehat{F}} \right|_{\cC_N^{\infty}(P)}. 
\end{equation}
Throughout this section we assume the above setting. We now recall, restricted to this setting, the pertinent results proven by Faure \& Tsujii~\cite{FT15}. Let \(\cD_N'(P)\) denote the dual of \(\cC^\infty_N(P) \).

\begin{QuasiComp}
\label{thm:QuasiComp}
 For any \(N\in \bZ\) and sufficiently large \(r>1\) there exists a Hilbert space \(\cH_N^r(P)\), satisfying
 \[
  \cC^\infty_N(P)  \subset \cH_N^r(P) \subset \cD_N'(P),
 \]
 such that \(\widehat{F}_N\) extends to a bounded operator
 \(  \widehat{F}_N :  \cH_N^r(P) \to  \cH_N^r(P)\)
with spectral radius bounded above by  \(\smash{\lambda}\) and essential spectral radius bounded above by \(\smash{\lambda^{-(r-1)}}\).
Moreover \(1\) is in the spectrum of \(\widehat{F}_0\), has multiplicity \(1\) and the associated eigendistribution corresponds to the invariant measure.
\end{QuasiComp}
\noindent
The Hilbert spaces \(\cH_N^r(P)\) are called \emph{anisotropic Sobolev spaces}. It is possible to obtain much more detailed information about the spectrum in this setting. Specifically, the eigenvalues of the transfer operators occur\footnote{\begin{samepage}Graphical representation of the spectral result which is described by \cite[Theorem~1.3.4]{FT15} (smooth case).  The circles drawn are of radius \(\lambda, \lambda^{{1}/{2}}, \lambda^{-{1}/{2}}, \lambda^{-{3}/{2}}, \lambda^{-{5}/{2}},\) etc. 

\nopagebreak
\vspace{1pt}
\noindent
\begin{minipage}{\textwidth}
\hspace{1cm}
\begin{tikzpicture}[scale=0.8]           
     \draw[color=gray!30!white, thin] (0,0) circle (1.8 cm);
     \foreach \r in {1.342,0.745,0.414,0.23,0.128,0.071}{
        \draw[color=gray!30!white, thin] (0,0) circle (\r cm);}
    \draw (-1.9,0) -- (1.9,0);
    \draw (0,-1.9) -- (0,1.9);
    \fill (0:1.8cm) circle (2pt);
    \fill (0:0cm) circle (2pt);
        \path (225:2.2cm) node (v0) {\(\widehat F_{0}\)};
\end{tikzpicture} 
\hfill
\begin{tikzpicture}[scale=0.8]           
     \draw[color=gray!30!white, thin] (0,0) circle (1.8 cm);
     \foreach \r in {1.342,0.745,0.414,0.23,0.128,0.071}{
        \draw[color=gray!30!white, thin] (0,0) circle (\r cm);
        \fill (20:\r cm) circle (2pt);
        \fill (28:\r cm) circle (2pt);
        \fill (200:\r cm) circle (2pt);
        \fill (260:\r cm) circle (2pt);
        \fill (100:\r cm) circle (2pt);
        \fill (110:\r cm) circle (2pt);
        \fill (175:\r cm) circle (2pt);
        \fill (80:\r cm) circle (2pt);
        \fill (310:\r cm) circle (2pt);
        \fill (350:\r cm) circle (2pt);}
    \draw (-1.9,0) -- (1.9,0);
    \draw (0,-1.9) -- (0,1.9);
    \path (225:2.2cm) node (v0) {\(\widehat F_{10}\)};
\end{tikzpicture} 
\hfill
\begin{tikzpicture}[scale=0.8]           
     \draw[color=gray!30!white, thin] (0,0) circle (1.8 cm);
     \foreach \r in {1.342,0.745,0.414,0.23,0.128,0.071}{
        \draw[color=gray!30!white, thin] (0,0) circle (\r cm);
        \fill (54:\r cm) circle (2pt);
        \fill (50:\r cm) circle (2pt);
        \fill (240:\r cm) circle (2pt);
        \fill (250:\r cm) circle (2pt);
        \fill (115:\r cm) circle (2pt);
        \fill (119:\r cm) circle (2pt);
        \fill (180:\r cm) circle (2pt);
        \fill (85:\r cm) circle (2pt);
        \fill (290:\r cm) circle (2pt);
        \fill (320:\r cm) circle (2pt);
        \fill (2:\r cm) circle (2pt);
        \fill (24:\r cm) circle (2pt);
        \fill (270:\r cm) circle (2pt);
        \fill (240:\r cm) circle (2pt);
        \fill (170:\r cm) circle (2pt);
        \fill (145:\r cm) circle (2pt);
        \fill (170:\r cm) circle (2pt);
        \fill (200:\r cm) circle (2pt);
        \fill (377:\r cm) circle (2pt);
        \fill (345:\r cm) circle (2pt);}
    \draw (-1.9,0) -- (1.9,0);
    \draw (0,-1.9) -- (0,1.9);
    \path (225:2.2cm) node (v0) {\(\widehat F_{20}\)};
\end{tikzpicture} 
\hspace{1cm}
\end{minipage} 
\vspace{1pt}
\nopagebreak
\noindent
Where exactly the eigenvalues lie on each circle is merely representative in the above image. However, in theory, it is possible to explicitly calculate them for the linear case~\cite{Faure07}. \end{samepage}
} in narrow bands centred on the circles of radius \(r_k=\lambda^{-\frac{1}{2}-k}\).

\begin{Bands}
\label{thm:Bands}
For any \(\epsilon >0\), there exists \(C_\epsilon>0\), \(N_{\epsilon} \geq 1\), such that, for any \(\abs{N} \geq N_{\epsilon}\), the spectrum of \(  \widehat{F}_N :  \cH_N^r(P) \to  \cH_N^r(P)\) is contained within the union of bands
\[
 \bigcup_{k \geq 0} \{z \in \bC \, : \, |z| \in (r_k - \epsilon, r_k+\epsilon)\}.
\]
Moreover, when \( r_k + \epsilon < \eta < r_{k-1} - \epsilon \) for some \(\eta>0\), \(k\geq 1\) then
\[
 \sup_{\abs{z}=\eta}\norm{\smash{(z - \widehat{F}_N)^{-1}}} \leq C_\epsilon.
\]
\end{Bands}
\noindent
Crucially, for our subsequent use, the resolvent bound included in this theorem is independent of \(N\in \bZ\).

\begin{rem}
While Theorem~\cite[1.3.4]{FT15} is presented for a larger class of systems, we state the above Theorem for the linear case, and we provide a specific application of this in Section \eqref{sec:Heisenberg}. The general results of Faure \& Tsujii \cite{FT15} provide a framework for future extensions of the results in Sections \ref{sec:ErgAver} and \ref{sec:Cohom} to a larger class of parabolic flows.
\end{rem}

\begin{rem}
 The spectral results we use, concerning partially hyperbolic maps, are available in higher dimensions, however the method we use in Section~\ref{sec:ErgAver} relies on the stable bundle being one dimensional (corresponding to the flow lines of the parabolic flow) in order to connect ergodic integrals to the norm. 
\end{rem}

\begin{rem}
 We restrict ourselves to the case where the unstable Jacobian is smooth. The more general case is possible (in general the foliations would only be  H\"older) and the analogous results are already available~\cite[\S1.4]{FT15} using the technique of  extending to the Grassmanian. The analogous method of using Grassmanian was applied by Giulietti \& Liverani~\cite{GL16} in the case of toral flows. For our application, we need the stable foliation to have some degree of regularity (since the leaves are flow lines for the parabolic flow) and a result of Hurder \& Katok~\cite{HK90} says that any smooth, volume-preserving Anosov map of \(\bT^2\) with \(\cC^{1+\omega}\) invariant foliations is smoothly conjugate to a linear Anosov toral automorphism. (The notation \(\cC^{1+\omega}\) means differentiable with the derivative being in the class given by the modulus of continuity \(\omega(s) = \operatorname{o}(s \abs{\log (s)}) \).) 
\end{rem}
 
In the setting described in this section, under the assumption that the outer band is isolated from the other bands, an estimate on the number of eigenvalues is obtained.
\begin{NumRes}
\label{thm:NumRes}
Counting multiplicity, the number of eigenvalues of \(\hat{F}_N\) on the outer band \(\{ \abs{z} \in (\lambda^{1/2}-\epsilon, \lambda^{1/2}+\epsilon)\}\) is equal\footnote{The symplectic volume of \(\cS\) is \(\Vol_\omega(\cS) = 	\int_\cS \omega  \).} to \(N \Vol_\omega(\cS) + \cO(1)\).
\(\hat{F}_N\) has  \(\approx N\) resonances in the outer band.
\end{NumRes}
\noindent
We know from the Franks-Newhouse Theorem \cite{Franks70, Newhouse70}  that an Anosov diffeomorphism of a two-dimensional compact Riemannian manifold (in particular co-dimension one) is topologically conjugate to a hyperbolic toral automorphism. The spectrum of the transfer operator for the specific case of linear hyperbolic maps on \(\bT^2\) has been studied in Faure \cite{Faure07}, and the resonances are shown to precisely lie on circles of radius \(r_k.\) In this specific setting, 
\(\hat{F}_N\) has exactly \(N\) eigenvalues (counting multiplicity) on each circle of radius \(r_k\).
Moreover, it is shown~{\cite[Thm\,1]{Faure07}} that the resonances on different circles have the same phases and differ only in modulus by some power of \(\lambda\).

\begin{rem}
Observe that the outer band of the spectrum lies outside of the unit circle and, since \( {\widehat{F}} = \sum_{N\in \bZ}  {\widehat{F}}_N\), we will inevitably have to manage a countable\footnote{In \cite[Rem.\,2.15]{GL16}, it was written that the case of countable deviation spectrum corresponds to flows whereas the  case of finite spectrum corresponds to maps. Instead we note here that the presence of a neutral direction (e.g., in the present work as well as when studying flows) is the distinguishing factor which determines the unavoidable presence of a countable number of eigenvalues in the problem.} number of eigenvalues if we want to obtain explicit  properties from the spectral information.
\end{rem}

\section{Additional Functional Analytic Information}
\label{sec:ExtraFunc}
This section is devoted to the details which are required in order to use the results from Section~\ref{sec:Spectrum} for the estimates related to the deviation of ergodic averages and to the cohomological equation problem.
First impressions may suggest serious issues using the spectral information from \( \widehat F_N\) as the spectrum of \( {\widehat{F}} = \sum_{N\in \bZ}  {\widehat{F}}_N\) is dense~\cite[Thm.\,1.3.11]{FT15} on each circle of radius \(\lambda^{\frac{1}{2}-k}\), \(k\in\{0,1,2,\ldots\}\). Nevertheless there is sufficient structure in order to deduce useful and explicit information; in particular the spectrum of \(\widehat{F}_N \) can be used to produce a (countable) spectral decomposition of the operator  \(\widehat{F} \) on a relevant space. Additionally, adequate estimates indicating that this norm is applicable to estimating smooth ergodic integrals must be obtained.

Since \( \cC^\infty(P) =\bigoplus_{N\in\bZ}  \cC^\infty_N(P)   \) it is convenient to define, for  \(\kappa>1\) (to be fixed shortly), the norm
\[
 \norm{h}_{\cH} := \sup_{N\in\bZ} \abs{N}^\kappa  \norm{h_N}_{\cH^r_N}
\]
where \(h = \sum_N h_N\) and \(h_N\in \cC^\infty_N(P)\) for each \(N\in \bZ\).

\begin{lem}
 \label{lem:BoundedHr}
 If \(h \in \cC^\infty_N(P)\) then  \(\norm{h}_{\cH} <\infty\).
\end{lem}
\begin{proof}
 Since \(\cH^r(P) \) is a generalised Sobolev space with an anisotropic weight that is bounded polynomially in frequency we know that  \cite[(1.6.5)]{FT15} there exists \(\nu \in \bN\) (indeed any \(\nu\) sufficiently large suffices) such that, for any  \(h \in \cC^\infty_N(P)\), there exists \(C_\nu >0\) such that  \(\norm{h_N}_{\cH^r_N} \leq C_\nu \abs{N}^{-\nu}\) for all \(N \in \bN\) .
 This means, as long as \(\nu\) is chosen sufficiently large, the supremum in the definition of the norm is finite.
\end{proof}

We define the weaker norm
\[
 \norm{h}_{\widetilde\cH} := \sup_{N\in\bZ} \abs{N}^{\kappa-1}  \norm{h_N}_{\cH^r_N}.
\]
 The two Banach spaces \(\cH(P) \subset \widetilde\cH(P) \) are then defined as the completion of \(\cC^\infty(P)\) with respect to, respectively, the norms \( \norm{\cdot}_{\cH}\), \( \norm{\cdot}_{\widetilde\cH}\).

\begin{bigthm}
 \label{thm:SpecInfo}
For all \(\eta>0\) there exists \(r>0\) and,
for each \({N}\in \bZ\) there exists \(K_N\in \bN\) and a set of eigenvalues \({\{\xi_{N,j}\}}_{j=1}^{K_N}\), projectors \({\{P_{N,j}\}}_{j=1}^{K_N}\) and nilpotents \({\{Q_{N,j}\}}_{j=1}^{K_N}\) such that the operators \(  \widehat{F}_N :  \cH_N^r(P) \to  \cH_N^r(P)\) satisfy the decomposition
\begin{equation}
\label{eq:DecomA}
 \widehat{F}_N
=
\sum_{j=1}^{K_N} \xi_{N,j}P_{N,j}  + Q_{N,j}  +  \widehat{F}_N P_{N,0}.
\end{equation}
The projectors and nilpotents satisfy the commutation relations: \(P_{N,j}P_{N,k} = \delta_{j,k}\), \(P_{N,j}Q_{N,k} =Q_{N,k}P_{N,j} = \delta_{j,k}Q_{N,k}\) and \(P_{N,0}\) is the projector corresponding to the part of the spectrum contained within \(\{ \abs{z}\leq \eta \}\). Viewing the operators defined on  \(\cH_N^r(P)\) as operators on \(\cH(P)\), the operator \(\widehat{F}: \cH(P) \to \widetilde\cH(P)\) satisfies the decomposition
\begin{equation}
\label{eq:DecomB}
\widehat{F}
=
\sum_{{N}\in \bZ}
\sum_{j=1}^{K_N} \xi_{N,j}P_{N,j}  + Q_{N,j}    +  \widehat{F} P_0
\end{equation}
where \(P_0 =\sum_{N\in \bZ} P_{N,0}\).
Moreover: (1) There are only a finite number of eigenvalues \(\xi_{N,j} \) with absolute value greater than \(\lambda^{\frac{1}{2}}+\epsilon\); (2) There exists \(C_\eta>0\) such that, for all \(n\in \bN\), 
\[
 \norm{\smash{\widehat{F}^n  P_{0}}}_{\widetilde\cH}
 \leq 
 \eta^n C_\eta.
\]
\end{bigthm}
\noindent
The above spectral information corresponds to the separation of the spectrum into the part outside of the circle \(\abs{z}=\eta\) and the remainder inside the circle \(\abs{z}=\eta\). We refer to \(\eta\) as the magnitude of the remainder term.

\begin{rem}
Viewing the operator as an operator acting from a stronger to a weaker space is a convenience since this result suffices entirely for our present purposes and requires relatively weak assumptions. There is some similarity to the idea which can be used for flows~\cite{Butterley16} when oscillatory cancellation results are not as strong as could be desired. However in present work the motivation is different in the sense that good oscillatory cancellation estimates are already available~\cite{FT15} but now we wish to go deeper into the spectrum.
\end{rem}

\begin{rem}
In reality, since our case reduces to the case of studying lifts of toral automorphisms (i.e., we are in the algebraic case), we know \cite{Faure07} that this ``finite number of eigenvalues'' in the above statement is actually a singleton consisting only of the eigenvalue at \(\lambda\) which corresponds to the measure of maximal entropy (in the linear case this coincides with the SRB measure).
\end{rem}

\begin{proof}[Proof of Theorem~\ref{thm:SpecInfo}]
Choosing \(\epsilon>0\) small  and \(r\) large (observe that \(\epsilon>0\) and \(r>0\) in Theorem~\cite[1.3.4]{FT15} are independent of \(N\)) the quasi compactness result of Theorem~\cite[1.3.1]{FT15} implies immediately that we have the spectral decomposition~\eqref{eq:DecomA} such that the finite sum of terms covers all of the spectrum outside of the set \(\abs{z}\leq \eta\). Note that the order of the nilpotents is not greater than \({K_N}\). Choosing \(\eta>0\) slightly smaller if required, we can assume that \(r_k + \epsilon< \eta < r_{k-1} -\epsilon \) and consequently  \(P_{N,0} = -\frac{1}{2\pi i}\int_{\abs{z}=\eta} R_N(z) \ dz\) where \(R_N(z) := \smash{(z - \widehat{F}_N)^{-1}} \) is the resolvent operator.
 
In order to meaningfully consider iterates we need to establish control on \( \norm{\smash{ {\widehat{F}}_N^n  P_{N,0}}} \). Observe\footnote{The case \(n=0\) is immediate from the previously mentioned integral formula for \(P_{N,0}\). That \(R(z) =-\sum_{n=0}^{\infty} z^{-(n+1)}  {\widehat{F}}_N^n  \) allows us to show that \(  {\widehat{F}}_N^n R_N(z) = zR_N(z) + \id \). This suffices to show that if the equality is true for \(n\) then it is true for \(n+1\) and consequently for all \(n\in\bN\). } 
that, for all \(n\in \bN\),
\[
 \widehat{F}_N^n  P_{N,0}
 =
 - \frac{1}{2\pi i}\int_{\abs{z}=\eta} z^n R_N(z) \ dz
 \]
and so the (uniform in \(N\)) control on \(\sup_{\abs{z}=\eta} \norm{R(z)} \) given by Theorem \cite[Theorem~1.3.4]{FT15} implies that, as an operator \(\smash{\widehat{F}_N^n  P_{N,0}} : \cH^r_N(P) \to \cH^r_N(P) \),
\[
 \norm{\smash{\widehat{F}_N^n  P_{N,0}}}
 \leq 
 \eta^n C_\epsilon.
\]
The uniformity, in \(N\), of this estimate means that the same estimate holds for \(\norm{\smash{\widehat{F}^n  P_{0}}}_{\cH}\).

This is not the only issue related to the spectral representation having a countable but not finite sum. Finally we must deal with the fact that \(K_N\) is not bounded but grows with \(\abs{N}\). Here we use the fact that we consider the full decomposition as an operator between a stronger and weaker space. We already know that \(K_N\) is bounded above by a constant which is proportional to \(\abs{N}\) (and by the number of bands being considered but this is constant once \(\eta>0\) is fixed). This, by a simple estimate of the sum and the definition of the \(\cH\) and \(\widetilde\cH\) norms guarantees that the countable sum converges. 
\end{proof}

The final functional analytic ingredient is to show that certain objects are elements of the dual and obtain good bounds for these.  
This is a notion of compatibility of the norm and is essential in order to use the spectral results to study the pertinent questions related to nilflows. (Such an estimate corresponds to \cite[Lem.\,5.11]{Adam18} and \cite[Lem.\,B.4]{GL16}.)
Given a finite length piece of stable manifold \(\gamma \subset P\) (stable manifold of \(\tilde f:P \to P \)) we define, for any \(\varphi \in \cC_0^q(\gamma)\), \(h\in \cC^\infty(P)\),
\begin{equation}\label{eq:stablefunctional}
 H_{\varphi}(h) = \int_\gamma \varphi \cdot h.
\end{equation}
\begin{lem}
\label{lem:InDual}
For any $\nu\in \bN$ there exists $C_\nu >0$ such that, for all $h \in \cC^\infty(P)$, \(N\in \bN\), 
stable curve \(\gamma\) of unit length, $\varphi \in \cC_0^\nu(\gamma)$, 
\[
\abs{H_{\varphi}(h)} \leq C_\nu  \norm{
h}_{\cH^r_N}\norm{\varphi}_{\cC^{\nu}(\gamma)}.
\]
\end{lem}
%
In order to prove the lemma we first recall the pertinent details of the definition of the anisotropic Sobolev spaces. Then we 
prove a version of  Lemma~\ref{lem:InDual} applicable when \( \abs{\gamma}=O(\abs{N}^{-1})\). Lastly, we combine the estimates for pieces of length \(\abs{N}^{-1}\) of the stable curve to obtain the desired estimate for when \(\abs{\gamma}=1\).

Using charts to work locally and observing that for each fixed \(N\) the behaviour in the neutral direction is predetermined, the definition of the norm~\cite[Def.\,6.4.1]{FT15}  on the manifold is reduced to defining\footnote{The reference allows for higher dimension but here we restrict to the case (\(d=1\)) as required in this present work. Although the reference uses interchangeable \(N\) and \(\hbar\) where \(\hbar = \frac{1}{2\pi N} \)  here we systematically use \(N\). } the norm for functions in \(\cC^\infty(\bR^2)\).  The connection between the \emph{local data}~\cite[(6.3.1)]{FT15} used for defining the norm and the original observable is the natural one \cite[Prop.\,6.3.2]{FT15}. It is important to note that the charts are chosen to depend on \(\abs{N}\) and the size of the charts is of order \(\abs{N}^{-1}\).

The Bargmann transform  \(\cB_N : L^2(\bR^2) \to L^2(\bR^4)\)  is  defined \cite[Def.\,3.1.1]{FT15} as
\( (\cB_N h)(x,\xi) = {( h, \phi_{x,\xi})}_{L^2(\bR^2)}\) 
where the Bargmann kernel is defined as 
\[
 \phi_{x,\xi}(y) := \sqrt{2N} \exp\left( {\pi i N}\xi \cdot (2y-{x}) - {\pi N}\abs{y-x}^2 \right).
 \]
Then, using a simple scaling  \cite[(4.2.7)]{FT15}, the operator \(\cB_x : L^2(\bR^2) \to L^2(\bR^4)\) is defined as \(\cB_x := \tilde\sigma^{-1} \circ \cB_N \circ \sigma\) where \(\sigma h(x) := 2^{-\frac{1}{4}}h(2^{-\frac{1}{2}}x)\) and \(\tilde \sigma v(x,\xi) := v(2^{-\frac{1}{2}}x, 2^{\frac{1}{2}}\xi)\).

The escape function  is \(\cW_N^r(x,\xi):=W_N^r(\zeta_p,\zeta_q)\). Here \((\zeta_p,\zeta_q)\)  are the coordinates \cite[(4.2.5)]{FT15} in phase space (\(\xi\)) which are called \emph{normal coordinates} \cite[Prop.\,2.2.6]{FT15} and correspond to the stable and unstable dynamics. 
On the other hand  \cite[Def.\,3.3.2]{FT15} \(W_N^r(\zeta_p,\zeta_q) := W^r( {\scriptstyle \sqrt{2 \pi N} } \zeta_p,  {\scriptstyle\sqrt{2 \pi N}} \zeta_q)\) is the anisotropic weight function based on the fixed function \( W^r(\zeta_p,\zeta_q)\) which is defined \cite[\S3.3]{FT15} to be smooth and, in particular, such that \( W^r(\zeta_p, 0) \approx \abs{\zeta_p}^{-r}\) and   \( W^r(0, \zeta_q) \approx \abs{\zeta_q}^{r}\) when \( \abs{\zeta_p}\), \( \abs{\zeta_q}\) are large.

Using these two ingredients, the norm (the local version used within charts) is defined \cite[Def.\,4.4.1]{FT15} as
\[
\norm{h}_{\cH^r_N(\bR^2)} := \norm{ \cW^r_N \cdot \cB_x h }_{L^2(\bR^4)}.
\]

\begin{lem}
\label{lem:FirstBound}
For any $\nu\in \bN$ there exists $C>0$ such that, for all \(N\in \bZ\), $h \in \cC^\infty_N(P)$, 
stable curve \(\gamma\) of length \(\abs{N}^{-1}\) and $\varphi \in \cC_0^\nu(\gamma)$, 
\[
\abs{H_{\varphi}(h)} \leq C  \norm{
h}_{\cH^r_N(P)}\norm{\varphi}_{\cC^{\nu}(\gamma)}.
\]
\end{lem}
\begin{proof}
We can assume that we are working within a single chart and, using local data \cite[(6.3.1)]{FT15}, that we are considering the integral over \(\tilde \gamma \subset \bR^2\) where \(\tilde \gamma\) is the projection of \(\gamma \subset P\) onto \(\mathbb{R}^2\). Abusing notation we will now consider \(h\) to be the local data.
We will estimate \(\tilde\varphi(h)  =\int_{\tilde\gamma} \varphi \cdot h\).
Considering \(\tilde \varphi\) as a distribution on Schwarz space we wish to estimate \(\abs{\tilde\varphi(h)}= \abs{\smash{ {\langle \tilde\varphi, h\rangle}_{L^2(\bR^2)}}}\).
Using the property of the Bargmann transform, \(\cB^*_N \cB_N =id\), shown in \cite[Prop.\,3.1.4]{FT15} (note that \(\cB_N \cB_N^*\) is not the identity but instead the orthogonal projection onto the image of \(\cB_N\) in \(L^2(\bR^4)\)) and introducing the escape function \cite[Def.\,3.3.2]{FT15} we can write
\[
{\langle \tilde\varphi, h\rangle}_{L^2(\bR^2)}
= {\langle \tilde\varphi, \cB_x^*\cB_x h\rangle}_{L^2(\bR^4)}
= {\langle\tfrac{1}{\cW^r_N} \cdot \cB_x\tilde\varphi, \cW^r_N \cdot \cB_x h\rangle}_{L^2(\bR^4)}.
\]
By definition of the norm,
\(\norm{h}_{\cH^r_N(\bR^2)} = \norm{ \cW^r_N \cdot \cB_x  h }_{L^2(\bR^4)} \) we know that
\[
 \abs{\smash{ {\langle \tilde\varphi, h\rangle}_{L^2(\bR^2)}}}
\leq 
\norm{\smash{\tfrac{1}{\cW^r_N}} \cdot \cB_x\tilde\varphi}_{L^2(\bR^4)}
\norm{h}_{\cH^r_H(\bR^2)} 
\]
and consequently, in order to complete the estimate of the lemma, it suffices to estimate \(\norm{\smash{\tfrac{1}{\cW_N^r}} \cdot \cB_x\tilde\varphi}_{L^2(\bR^4)}\).

 We can assume that we are working in coordinates \cite[(4.2.5)]{FT15} such that \(y = (y_p, y_q)\) and that \(\gamma = \{(y_p, y_q) : y_p\in (-1,1), y_q=0\}\).
We calculate that \( (\cB_N \tilde\varphi)(x,\xi) \) is equal to
\begin{multline*}
\sqrt{2N}
\int_{-1}^{1} \varphi(y_p)   \exp\left(  \pi i N
\left(\begin{smallmatrix}
\xi_p \\ \xi_q
\end{smallmatrix}\right)
 \cdot 
 \left(\begin{smallmatrix}
2y_q-{x_p} \\ -{x_q}
\end{smallmatrix}\right)
- \pi N \abs{ \left(\begin{smallmatrix}
y_p-{x_p} \\ -{x_q}
\end{smallmatrix}\right)}^2 \right) \ dy_p
\\
= 
\sqrt{2N}
\int_{-1}^{1} \varphi(y_p)   \exp\left( \pi i N
\xi_p 
 \left(2y_p-{x_p} \right)
-\pi  N \abs{ 
y_p-{x_p} }^2 \right) \ dy_p
\\
\times
\exp\left( -\pi i N   \xi_q 
{x_q}
- \pi  N { {x_q} }^2 \right)
\end{multline*}
The smoothness of \(\varphi\), together with integration by parts, implies the decay in \(\abs{\xi_p}\). Coupling this with basic estimates we obtain (considering \(\varphi \in \cC^\nu\))
\[
\abs{(\cB_N \tilde\varphi)(x,\xi)}
\leq
C_\nu
\exp \left(-{\pi  \abs{N}} (x_p^2 + x_q^2 ) \right)  \abs{N \xi_p}^{-\nu}.
\]
To finish we recall that, because of the anisotropic estimates available for the escape function, 
\(\tfrac{1}{\cW_N^r}(x,\xi) \leq C\left(\abs{\xi_p}^r \abs{N}^\frac{r}{2}   +  \abs{\xi_q}^{-r} \abs{N}^{-\frac{r}{2}} \right)\).
We require \(\nu < r\) and thus, by multiplying the two estimates, we have shown that \(\norm{\smash{\tfrac{1}{\cW_N^r}} \cdot \cB_x\tilde\varphi}_{L^2(\bR^4)}\) is bounded, independently of \(N\).
\end{proof}

\begin{proof}[Proof of Lemma \ref{lem:InDual}]
It is now a simple task to combine the estimate of Lemma~\ref{lem:FirstBound} in order to prove the lemma. Firstly we must divide the stable curve \(\gamma\) into a number of pieces of length \(\abs{N}^{-1}\) using a partition of unity in order to guarantee that each piece can be studied within a single chart (recall that for larger \(\abs{N}\) finer charts are used). The number of pieces will be of order  \(\abs{N}\). Secondly we write \(h=\sum_{N\in \bZ} h_N\) where \(h_N \in \cC^\infty_N(P)\). This generates a countable sum of terms; we again take advantage of the decay guaranteed by the definition of the \(\norm{\cdot}_{\cH}\) norm.
\end{proof}

\section{Deviation of Ergodic Averages}
\label{sec:ErgAver}

In this section we use the information described in Section~\ref{sec:ExtraFunc}, in particular Theorem~\ref{thm:SpecInfo},  to give a precise description of the deviation of ergodic averages for the parabolic flow.  A similar concept for the study of solutions of the cohomological equation is postponed until Section~\ref{sec:Cohom}. The standing assumption of this section is that we have a flow \(\Psi_t \) which is renormalized~\eqref{eq:RenormFlow} by a partially hyperbolic map \(F\). We further require that the transfer operator associated to \(F\) has a spectral decomposition as given in the conclusion of Theorem~\ref{thm:SpecInfo} with respect to a function space (in the extended sense) which satisfies the conclusion of Lemma~\ref{lem:InDual}. This section is completely independent of the precise details of the function space being used, as long as it satisfies the above mentioned properties.
%

In the first part of this section we will estimate precisely the ergodic integral~\eqref{eq:ErgodInt} \(H_{x,t}(h)\) and prove the following.
\begin{bigthm}
\label{thm:Deviation}
Let \(\epsilon>0\), \(\beta \in (-\infty, 1)\) and let \({\{\xi_j\}}_{j=1}^{\infty}\) and  \({\{P_j\}}_{j=1}^{\infty}\) denote the countable set of eigenvalues and corresponding projectors given by the spectral decomposition of Theorem~\ref{thm:SpecInfo} such that, letting
\[
\alpha_j := \tfrac{\ln |\xi_j|}{h_{\textnormal{top}}} \in (\beta,1),
\]
there exists \(J\in\bN\) such that \(\abs{\xi_j} \leq \lambda^{\frac{1}{2} + \epsilon}\) for all \(j\geq J\) and that the magnitude of the remainder term is bounded by \(e^{\beta \htop }\).\\
(i) There exist linear functionals \(\ell_{x,j}^t(\cdot)\) such that, for all \(x \in M \), \(t \in \bR\), and \(h \in \cC^\infty(M)\), 
\[
H_{x,t}(h) = 
t  \int_M h \ dm 
+
\sum_{j=1}^{\infty}
t^{\alpha_j}(\ln t)^{d_j}
\ell_{x,j}^t(h)
+ \cO\left(\smash{t^{\beta}\norm{h}_{\cC^r}}\right).
\]
The linear functionals are bounded uniformly in \(t\) and are related to the projectors in the sense that if \(h\in \ker (P_j)\) then \(\ell_{x,j}^t(h) =0\).\\
(ii) In particular, if \(h \in \ker(P_j)\) for all \(j \leq J\) then there exists \(C >0\) such that, for all \(t\geq 0\),
\[
\abs{H_{x,t}(h)} \leq C t^{\frac{1}{2}+ \epsilon}.
\]
\end{bigthm}  

\noindent
The proof of this theorem follows after several useful lemmas and includes the precise definition of the linear functionals.
Recall that the transfer operator  \eqref{eq:TransOp} was defined as \(\hat{F} : h \mapsto \lambda h\circ F^{-1}\). 
The basic idea is that we use the renormalization \eqref{eq:RenormFlow} to transform the ergodic integral \eqref{eq:ErgodInt} 
\(H_{x,t}(h) =  \int_0^t  h \circ \Psi_{r}(x) \ dr\) as shown in the computation \eqref{eq:IntRenorm}.
In this way the long time behaviour is understood by studying iterates of the map;
it is natural to take \(k=k(t)\) such that \(\lambda^k t\) is of the same scale for all \(t\).
In order to compute estimates in terms of the spectral information of the transfer operator, we want to work with continuous functionals so we would like to have smooth versions of ergodic integrals.
To this end, for any $\varphi \in \cC^{\infty}(\mathbb{R})$ with compact support, \(x \in M\), and \(h \in C^0(P)\), let
\begin{equation}\label{eq:smoothergint}
H_{x, \varphi}(h) := \int_{\mathbb{R}} \varphi(t) \cdot h \circ \Psi_{t}(x)\> dt.
\end{equation}
Note that the smooth version of the ergodic integral  \eqref{eq:smoothergint} corresponds to the functional \eqref{eq:stablefunctional} where the integral is taken over a stable curve.   Lemma~\ref{lem:InDual} means that this quantity is well controlled by the functional analytic results we have.
In the following lemma, we combine the renormalization \eqref{eq:RenormFlow} with the smoothed version of the ergodic integral \eqref{eq:ErgodInt}.

\begin{lem}
\label{lem:RenormH}
Suppose that $k \in \mathbb{Z}$. 
Then, for all $h\in \cC^0(P)$,  $x\in P$,
\[
H_{x,\varphi}(h) 
= 
H_{\tilde x,\tilde{\varphi}}(\hat{F}^kh)
\]
where 
$\tilde{\varphi}( t ) := \varphi(\lambda^{k} t )$ 
and
\( \tilde x := F^k(x) \).
\end{lem}

\begin{proof}
Using the renormalization \eqref{eq:RenormFlow}
\[
\begin{split}
H_{x,\varphi}(h)
= & \> \int_{\mathbb{R}} \varphi(r) \cdot h \circ F^{-k} \circ F^{k} \circ \Psi_{r}(x) \> dr\\
= & \> \int_{\mathbb{R}} \varphi(r) \cdot h \circ F^{-k}\circ \Psi_{\lambda^{-k}r}\circ F^{k}(x)  \> dr\\
= & \> \int_{\mathbb{R}} \varphi (\lambda^{k} u) \cdot \left( \lambda^{k} h \circ F^{-k} \right)  \circ \Psi_{u}(F^{k}(x))  \> du. 
\end{split}
\]
Since \( \lambda^{k} h \circ F^{-k} = \hat{F}^kh\) this completes the proof of the lemma.
\end{proof}

The next step is to relate the original ergodic integrals \eqref{eq:ErgodInt} with smooth versions. 

\begin{lem}
\label{lem:SmoothInt}
Suppose that \(T > 0\). 
Let \(N=\left\lfloor\frac{\ln T}{\ln 4} \right\rfloor \)
and for all \(k\in \bN\), \(\abs{k}\leq N  \), let 
\( m_k := \left\lfloor\left(\ln T - \abs{k} \ln 4 \right) / (\ln \lambda )\right\rfloor\).
There exists a set of functions \( {\{\tilde{\varphi}_k\}}_k \subset \cC^\infty(\bR, [0,1])\), each supported on an interval of unit length with the property that \( \norm{\tilde\varphi_k}_{\cC^q} \leq C_q\) (independent of \(T\)), such that, for every \(x\in P\) and every \(h\in \cC^\infty(P)\), letting \(x_k = F^{m_k}x\),
\[
H_{x,T}(h) = \sum_{k=-N}^{N}  H_{x_k,\tilde{\varphi}_k}(\hat{F}^{m_k}h) + \cO(\norm{h}_{\cC^0(M)}).
\]
\end{lem}
\noindent
Prior to proving this Lemma, it is convenient to introduce a partition of unity which ``zooms in'' on the end points of an interval. The construction is for an interval of length \(T> 0\), and the number of components of the partition is determined by \(n\in \bN\). 
\textbf{Step 1:} 
Fix, once and for all, a smooth step function \(\eta \in \cC^\infty(\bR, [0,1])\) such that \(\eta(t)=0\) whenever \(t\leq 0\) and \(\eta(t)=1\) whenever \(t\geq 1\). 
\textbf{Step 2:} 
For all \(k\in \{0,2,\ldots, n-1\}\) let 
\begin{equation}
\label{eq:DefPOU}
\eta_k(t) := \eta\left( \tfrac{8t-4^{-k}T}{4^{-k} T}  \right).
\end{equation}
\textbf{Step 3:} 
Let \(\varphi_0(t) := \eta_0(t) - \eta_0(T-t)\). 
For all \(k\in \{1,2,\ldots, n-1\}\), let \(\varphi_k := \eta_k - \eta_{k-1}\) and let \(\varphi_{n}:= 1 - \eta_{n-1}\); 
for all \(k\in \{1,2,\ldots, n\}\), let \(\varphi_{-k}(t) := \varphi_{k}(T-t)  \). We denote by \(\{\varphi_k\}_{k=-n}^n\) the zooming partition of unity.
Observe that \(\varphi_{n}\) and  \(\varphi_{-n}\) differ from the other  \(\varphi_{k}\) in that they are only one-sided bumps whereas all the other \(\varphi_k\) are compactly supported smooth functions.
This construction has the following properties.
\begin{lem}
\label{lem:ZoomPOU}
Let \( {\{\varphi_k\}}_{k=-n}^{n} \) be the zooming partition of unity for \([0,T]\) with \(n \in \bN\) components. Then
\begin{enumerate}
\item 
\( \sum_{k=-n}^{n} \varphi_{k}(t) = 1 \) for all \(t\in[0,T]\);
\item
The support of \({\varphi_k} \) is contained in an interval of length \( 4^{-\abs{k}} T \);
\item
Let \(\sigma_k: t \mapsto t4^{-\abs{k}}T\).
For each \(q\in \bN\) there exists \(C_q>0\), independent of \(T\), such that
\( \norm{\varphi_k \circ \sigma_k}_{\cC^q} \leq C_q\) for all \(\abs{k} \leq n\).
\end{enumerate}
\end{lem}
\begin{proof}
That (1) holds is immediate from the telescoping construction. 
For (2), observe that \(\varphi_0\) is supported on \((\frac{T}{8},\frac{7T}{8}) \). 
In the cases \(0<k < n\) the function \(\varphi_k\) is supported on the interval \((\frac{1}{8}4^{-{k}} T, 4^{-{k}} T) \) and the length of this interval is bounded by \( 4^{-\abs{k}} T \). The same bound also holds for \(k=n\). By symmetry when \(k\) is negative the supports  satisfy these same bounds.
The control on the derivatives, (3) is a consequence of the scaling used in the definition \eqref{eq:DefPOU} of the functions \(\norm{\varphi_k \circ \sigma_k}_{\cC^q} \leq 8^q \norm{\eta}_{\cC^q} \leq C_q\).
\end{proof}
We now use the above partition of unity for the postponed proof.
\begin{proof}[{Proof of Lemma \ref{lem:SmoothInt}}]
Let \(N := \left\lfloor\frac{\ln T}{\ln 4} \right\rfloor  \).
Using the zooming partition of unity, Lemma~\ref{lem:ZoomPOU}, we write
\[
\begin{aligned}
H_{x,T}(h) 
&=
\int_0^T h\circ \Psi_{t}(x) \ dt\\
&= 
\sum_{k=-(N-1)}^{N-1} H_{x,\varphi_{k}}(h)
+ 
\int_{\bR}  (\varphi_{-N} + \varphi_{N})(t)  \cdot  h\circ \Psi_{t}(x) \ dt.
 \end{aligned}
\]
The final integral term is bounded by \( 2 \cdot  4^{-{N}} T \norm{h}_{\cC^0(M)} \leq 2 \norm{h}_{\cC^0(M)} \). For the smooth ergodic integrals in the sum we apply Lemma~\ref{lem:RenormH} and obtain 
\[
 H_{x,\varphi_{k}}(h)
 =
 H_{x_k,\tilde\varphi_{k}}(\hat{F}^{m_k}h)\]
 where \( \tilde\varphi_{k}(t) := \varphi_{k}(\lambda^{m_k}t)  \), \(x_k =F^{m_k}(x)\).
Observe that, if we set  \( m_k := \left\lfloor\left(\ln T - \abs{k} \ln 4 \right) / (\ln \lambda )\right\rfloor  \), then
\[
\lambda^{-1} T 4 ^{-\abs{k}} < \lambda ^{m_k} \leq T4^{-\abs{k}}.
\]
Consequently, \(\norm{\tilde{\varphi}_k}_{\cC^q} \leq C_q\) for some \(C_q\) which does not depend on \(k\), \(n\), or \(T\) as per (3) of  Lemma~\ref{lem:ZoomPOU}.
\end{proof}

\begin{proof}[Proof of Theorem~\ref{thm:Deviation}]
We use the spectral decomposition of Theorem~\ref{thm:SpecInfo} to obtain 
\[
\hat{F}^{m}h
=
\sum_{j=0}^{\infty} \left(\xi_{j}P_{j}  + Q_{j}   \right)^{m}h + R_{m}h.
\]
where \(R_m = \widehat{F}^m P_0\) and so \( \norm{\smash{R_m h}}_{\cH^{r,\kappa}} \leq  \eta^n C_\eta \norm{h}_{\cH^{r,\kappa}} \).
By Lemma~\ref{lem:SmoothInt} we know that
\(
H_{x,t}(h) 
= 
\sum_{k=-N}^{N}  H_{x_k,\tilde\varphi_k}\left(\hat{F}^{m_k}h\right) + \cO(\norm{h}_{\cC^0})
\)
and so
\[
\begin{aligned}
H_{x,t}(h) 
&= 
\sum_{k=-N}^{N} \sum_{j=0}^{\infty}  H_{x_k,\tilde\varphi_k}\left(\left(\xi_{j}P_{j}  + Q_{j}   \right)^{m_k} h\right)  \\ & \ \ 
+ \sum_{k=-N}^{N}  H_{x_k,\tilde\varphi_k}(R_{m_k} h)
+ \cO(\norm{h}_{\cC^0}).
\end{aligned}
\]
Recall that \(m_k=\left\lfloor\frac{\ln t - \abs{k} \ln 4}{ \ln \lambda} \right\rfloor \) and so, in the case that \(\eta >1 \), observe that
\(\eta^{m_k} \leq C \exp \left(\frac{\ln \eta}{\ln \lambda}(\ln t - \abs{k} \ln 4)  \right) = C t^{\beta} 4^{-\beta \abs{k}}  \) where \(\beta = \ln \eta / \htop\). The case \(\eta\leq 1\) is elementary.
Combining Lemma~\ref{lem:InDual} and the estimate for \(R_p\) from Theorem~\ref{thm:SpecInfo}, 
\[
\begin{split}
\sum_{k=-N}^{N}  H_{x_k,\tilde\varphi_k}(R_{m_k} h) &
\leq 
\sum_{k=-N}^{N} C \eta^{m_k}   \norm{h}_{\cH^{r,\kappa}}.  \\
& \leq \sum_{k=-N}^{N} C t^{\beta} 4^{-\beta \abs{k}}   \norm{h}_{\cH^{r,\kappa}} 
\leq C' t^{\beta}  \norm{h}_{\cH^{r,\kappa}}.
\end{split}
\]
Let \({\{d_j\}}_j\) be the integers such that \(Q_j^{d_j} = 0\) and  \(Q_j^{d_j-1} \neq 0\) or \(d_j =0\).
Recall that we defined the linear functionals    
\[
\ell_{x,j}^t(h)
:=
t^{-\alpha_j}(\ln t)^{-d_j}
\sum_{k=-N}^{N}   H_{x_k,\tilde\varphi_k}\left(\left(\xi_{j}P_{j}  + Q_{j}   \right)^{m_k} h\right) . 
\]
It remains to show that \(\abs{\ell_{x,j}^t(h)}  \) is bounded in \(t\geq 0\). 
This follows from computing,
\[
\begin{aligned}
 \abs{ H_{x_k,\tilde\varphi_k}\left(\left(\xi_{j}P_{j}  + Q_{j}   \right)^{m_k} h\right)  }
&\leq
\xi_{j}^{m_k}{m_k}^{d_j} \norm{\tilde\varphi_k}_{\cC^1}  \norm{h}_{1}   \\
&\leq
C t^{\frac{\ln \xi_j}{\ln \lambda}} \xi_j^{-\abs{k}\frac{\ln 4}{\ln \lambda}}  (\ln t)^{d_j} \abs{\ln \lambda}^{d_j}   \norm{h}_{1}   \\
&\leq 
C t^{\alpha_j} (\ln t)^{d_j}
\end{aligned}
\]
where \(1 \leq \eta < \abs{\xi_j} \leq \lambda\), \(\ln \lambda = \htop\) is a constant, and \(\norm{\tilde{\varphi}_k}_{\cC ^1}\) is uniformly bounded (from Lemma \ref{lem:SmoothInt}). This concludes the proof of \((i)\).
To prove \((ii)\), observe that if \(h \in \ker(P_j)\) then \(\ell_{x,j}^t(h)=0\). So if  \(h \in \ker(P_j)\) for  each \(j \leq J\),
\(
H_{x,t}(h) \leq C \left (
\norm{h}_{\cC^0}
+ t^{\beta} \norm{h}_{\cC^r}\right ) \leq  C_h t^{\beta}\).
\end{proof}

\section{The Cohomological Equation}
\label{sec:Cohom}
In this section we show the existence of a solution to the cohomological equation. I.e., given a function \(h\) on \(P\) we wish to find (if it exists) \(g\) such that \(g\circ \Psi_t(x) - g(x) = \int_{0}^{t} h \circ \Psi_r(x) \ dr\)
for all \(x \in P\), \(t\geq 0\). If \(h\) is such that the ergodic integral $\int_0^t h\circ \Psi_s(x) \ ds$ is uniformly bounded in \(t>0\) then, by the Gottschalk-Hedlund Theorem, \(h\) is a coboundary. The argument we use now is an independent proof of this in this particular setting and also gives us the possibility for further investigation. Considering the differential version of the cohomological equation \(Wg = h\) where \(W\) represents differentiation with respect to the vector field associated to to the parabolic flow, we can rephrase our objective as finding an inverse to the operator \(W\). 

The method we use here follows closely the one used by Giulietti \& Liverani~\cite{GL16} although in places the present formulation is somewhat different. We will show that there is a formal inverse to the operator \(W\) that is determined by a countable sum. Expressing the coboundary in this way enables us to more easily explore the regularity.
Let \(\chi \in \cC^\infty(\bR^+,[0,1])\) such that \(\chi(s)=1\) for \(s \in [0, \frac{1}{2}]\) and \(\chi(s) = 0\) for \(s \in [1, \infty) \).
Let \(\varphi(t) := \chi(t/\lambda) - \chi(t)\) and observe that this is a bump function with support contained within \((\frac{1}{2},\lambda)\).  
For all \(h\in\cC^\infty(P)\) we define pointwise (this is a type of local integrating operator),
\begin{equation}
\label{eq:DefK}
(Kh)(x): =
 \int \varphi(s) \cdot h \circ \Psi_{s}(x)\> ds
\end{equation}
Observe that, by Lemma~\ref{lem:InDual}, this quantity is well defined on the anisotropic Sobolev spaces.
For any \(k\in \bN\),  \(h\in\cC^\infty(P)\) we again define pointwise 
\[
  \widetilde{G}h(x) :=  \int \chi(s) \cdot h \circ \Psi_{s}(x)\> ds,
  \quad
\left(G_kh\right)(x) :=  \left( K \widehat{F}^k h \right)(F^kx).
\]

\begin{bigthm}\label{thm:Cobound}
Let \({\{P_j\}}_{j\in\cJ}\) be a countable set of projectors given by Theorem~\ref{thm:SpecInfo} for the case when we choose \(\eta <1\).
Suppose that \(h \in \ker(P_j)\) for  all \(j\in\cJ\). 
Then the sum \(\sum_{k=0}^{\infty} G_k h(x)\) converges uniformly for \(x\in P\).
Moreover  \(h\) is a coboundary for the flow \(\{\Psi_{t}\}_t\) with transfer function
\[
g(x) = - \widetilde{G}h(x)  -  \sum_{k=0}^{\infty} G_k h(x).
\]
\end{bigthm}
\begin{rem}
 In a certain sense the formal inverse to differentiation by \(W\) is the operator defined as 
 \(G := - \widetilde{G}  -  \sum_{k=0}^{\infty} G_k\).
\end{rem}
\noindent
In order to prove Theorem \ref{thm:Cobound}, we first prove that the sum \( \sum_{k=0}^{\infty} G_k h(x)\) converges.
\begin{proof}[Proof of the first part of Theorem \ref{thm:Cobound}]
Since, by assumption, \(h \in \ker(P_j)\) for all \(j\in\cJ\), then the estimate of Lemma~\ref{lem:InDual}, together with Theorem \ref{thm:Deviation}, implies that
\[
\abs{ \left(G_kh\right)(x) }
= \abs{\left( K \widehat{F}^k h \right)(F^kx)}
  \leq C \norm{\smash{\hat{F}^{k}h}} \leq C \eta^{k} \norm{\smash{h}}.
\]
The sum is therefore bounded by a geometric sum and consequently converges. 
\end{proof}

\begin{lem}
\label{lem:SumEq}
For all \(n\in \bN\),  \(x\in M\) and \(h\in\cC^\infty(P)\),
 \[
  \int \chi(\lambda^{-n} t) \cdot h\circ \Psi_t(x) \ dt 
  =  \widetilde{G}h(x)  
  +  \sum_{k=0}^{n-1} G_k h(x).
 \]
\end{lem}
\begin{proof}
Let   \(t_j:= \lambda^{j}t\) for \( j \in \bN\).
We observe that, since \(\chi(\lambda^{-n} t)    =  \chi(s) +  \sum_{k=0}^{n-1}  \left[\chi(\lambda^{-(k+1)} t) -\chi(\lambda^{-k} t)   \right]  \),
\[
\begin{aligned}
 \int \chi(\lambda^{-n} t) \cdot h\circ \Psi_t(x) \ dt 
&=
\sum_{j=0}^{n-1} \int \left[\chi(\lambda^{-(k+1)} t) -\chi(\lambda^{-k} t)   \right] h\circ \Psi_t(x) \ dt
\\
& \ \ \ + \int \chi(s) \cdot h \circ \Psi_{s}(x)\ ds
\\
&=  
\sum_{k=0}^{n-1} \int \varphi(\lambda^{-k} t)   \cdot  h\circ \Psi_t(x) \ dt
+ \widetilde{G}h(x).
\end{aligned}
\]
Furthermore, using the renormalization of Lemma~\ref{lem:RenormH}, 
\[
  \int \varphi(\lambda^{-k} t)   \cdot  h\circ \Psi_t(x) \ dt =  \int \varphi( t)   \cdot  (\widehat{F}^{k}h) \circ \Psi_t(\widehat{F}^{k}x) \ dt = G_k h(x). 
\]
\end{proof}

\noindent
Having established that the limit exists we will now complete the proof of Theorem \ref{thm:Cobound} by showing that this limit is, in fact, the object we are looking for.
\begin{proof}[Proof of second part of Theorem \ref{thm:Cobound}]
We show that \(g\) satisfies 
\(g \circ \Psi_W^t(x) - g(x) = \int_0^t h \circ \Psi_{W}^{s}(x) \> ds\).
For convenience let  \[
g_n(x)  :=  \widetilde{G}h(x)  
  +  \sum_{k=0}^{n-1} G_k h(x).
 \]
For \(t \in (0, \frac{1}{2})\) (by the semigroup property this suffices to prove the result for all \(t>0\)), Lemma~\ref{lem:SumEq} implies that
\begin{multline*}
 g_n(x)  - g_n(\Psi_t x)
  - \int_0^t h \circ \Psi_{W}^{s}(x)\>ds\\
 = \int_{\bR^+} \left [ \chi(\lambda^{-n}s)- \mathsf{1}_{[0,t]}-\chi(\lambda^{-n}(s-t))  \right] h \circ \Psi_{W}^{s}(x) \> ds = H_{x, \eta_n}(h)
\end{multline*}
where we define \(\eta_n := \chi(\lambda^{-n}s)- (\mathsf{1}_{[0,t]}(s) + \chi(\lambda^{-n}(s-t)))\). Note that \(\supp(\eta_n) \subset (\frac{\lambda^{n}}{2}, \lambda^{n} + t)\). Using the renormalization of the smooth ergodic integral (\ref{lem:RenormH}), scaling by a factor \(\lambda^n\), we can write
\[
H_{x, \eta_n}(h)
=
\int_{s=t}^\infty \left[ \chi(s) - \chi(s-t) \right] \cdot (\widehat{F}^n)\circ \Psi_s(F^nx)\ ds
\]
which, by the repeatedly used argument,  exponentially small in \(n\). 
\end{proof}

\section{Heisenberg Niflows}
\label{sec:Heisenberg}

We describe an important special case to which the results of the previous sections apply.
For \(E \in \bN\), a \(3\)-dimensional {Heisenberg nilmanifold} 
\(M = \Gamma_E \setminus N\), is a compact quotient  of the Heisenberg group \(N\) by a lattice \(\Gamma_E\),  
\[
N := \left \{
\left(\begin{smallmatrix} 
   1      & x & z \\
 0      & 1 & y \\
 0       & 0 & 1
\end{smallmatrix}\right) : x,y,z \in \mathbb{R} \right \},
\]
\[
\Gamma_E := \left \{
\left(\begin{smallmatrix} 
   1      & p & \frac{r}{E} \\
 0      & 1 & q \\
 0       & 0 & 1
\end{smallmatrix}\right) : p,q, r \in \bZ \right \}.
\]
It is common to express elements of $N$ in terms of the entries above the diagonal; in this form, the group operation is given by\footnote{Folland \cite{Folland89} refers to this as the symplectic Heisenberg group. It is also common to see the polarized Heisenberg group, given by the group law \((x,y,z) * (a,b,c) = (x+a, y+b, z+c+xb)\), which corresponds directly to matrix multiplication. The two groups are equivalent; the map \((x,y,z) \to (x,y,z + \frac{1}{2}xy) \) gives an isomorphism between them.}
\[
(x,y,z) * (a,b,c) = \left(x + a, y + b, z + c + \frac{1}{2}(xb-ya)\right)
\]
where \(xb-ya=[(x,y),(a,b)]\) is the symplectic form on \(\bR ^2\). Since \(M\) is a nontrivial circle bundle over \(\bT^2\), we are precisely in the setting of Faure \& Tsujii~\cite{FT15} as described in Section \ref{sec:Spectrum}.
The standard basis of the Lie Alegbra of \(N\) is given by the elements
 \[  \left\{ X_{0}=
\left( \begin{smallmatrix}
  0 & 1 & 0 \\ 0 & 0 & 0 \\ 0 & 0 & 0\\
\end{smallmatrix}\right), \>
Y_{0}=\left(\begin{smallmatrix}
  0 & 0& 0 \\ 0 & 0 & 1\\ 0 & 0 & 0 \\
\end{smallmatrix}\right), \>
Z_{0}=\left(\begin{smallmatrix}
0 & 0 & 1 \\ 0 & 0 & 0 \\ 0 & 0 & 0\\
\end{smallmatrix} \right)\right\}
\] 
which satisfy the commutation relations
\begin{equation} \label{eq:LieBracket}
[X_{0},Y_{0}]=Z_{0} \> \> \textnormal{and} \> \>  [X_{0},Z_{0}]=[Y_{0},Z_{0}]=0.
\end{equation}
The natural probability measure, \(\mu\), on \(M\) is inherited from the Haar measure on \(N\).
Nilflows on \(M\) are given by the right action of one-parameter subgroups of \(N\), and thus, preserve \(\mu\). When the projected flow is an irrational linear flow on \(\bT ^2\), the nilflow is uniquely ergodic and minimal (see, e.g., \cite{FF06}). 
For \((x,y,z)\in M\), consider a partially hyperbolic automorphism of \(N\) given by 
\begin{equation}
\label{eq:phd}
F(x,y,z) = (A(x,y), z)
\end{equation}
where
\[
A = \begin{pmatrix} a & b \\ c & d 
\end{pmatrix} 
\]
is a \(2 \times 2\) matrix with integer entries, determinant 1, and eigenvalues \(\lambda, \lambda^{-1}\) for some \(\lambda^{-1} \in (0,1)\). 
\(F\) induces a diffeomorphism of \(M = \Gamma\setminus N\) when \(F(\Gamma)=\Gamma.\) Note that \(F\) preserves the symplectic form, and thus, can be considered in the framework of Faure \cite{Faure07} and Faure \& Tsujii~\cite{FT15} from Section \ref{sec:Spectrum}. 
We describe the relevant parabolic system that is renormalized by the automorphism \eqref{eq:phd} by constructing a flow in the stable direction. Consider an element of the Lie Alegbra, \(W = \alpha X_0 + \beta Y_0\), where \(\alpha,\beta\) are the non-zero entries of the normalised eigenvector associated to \(\lambda\).  We consider the flow generated by \(W,\) i.e., the right action of the subgroup of \(N\) given by\footnote{We can compute explicitly the formula \(\exp{tW} = (\alpha t, \beta t, \frac{1}{2} \alpha \beta t^2)\).}
\begin{equation}
\label{eq:Nilflow}
{\{ \Psi_{t}^{W} \}}_{t \in \bR} := \{\exp({tW})\}_{t \in \bR}.
\end{equation}
A simple calculation gives the following renormalization with the partially hyperbolic automorphism,
\begin{equation}
\label{eq:renormnilflow}
F \circ \Psi_{\lambda t}^W = \Psi_{t}^W \circ F.
\end{equation}
Combining Theorem \eqref{thm:SpecInfo} with the renormalization \eqref{eq:renormnilflow}, we can estimate the growth of ergodic integrals, in this case, of the form
\[
\int_{0}^{t} h \circ \Psi_r^W(x) \> dx
\]
by applying Theorem \ref{thm:Deviation}. We also get existence of a coboundary from Theorem \ref{thm:Cobound}. These results have already been obtained in Flaminio \& Forni \cite{FF06}.


\end{document}